\documentclass[twoside,a4paper,11pt,times,reqno]{amsart}
\usepackage{amsmath}
\usepackage[mathcal]{eucal}\usepackage{mathrsfs}
\usepackage{epsfig} \usepackage{psfrag}\usepackage{subfigure} 
\usepackage{tabularx}\usepackage{rotating}
\usepackage{amssymb}
\numberwithin{equation}{section}

\newtheorem{theorem}{Theorem}[section]
\newtheorem{lemma}{Lemma}[section]

\numberwithin{equation}{section}

\begin{document}

\title[Average Regularity of the Solution to an Equation]
{Average Regularity of the Solution to an Equation with the Relativistic-free Transport Operator}

\author{Jianjun Huang}
\address{Department of Mathematics, Sun Yat-Sen University, 
Guangzhou 510275, China}
\email{mljzsu@163.com}

\author{Zhenglu Jiang$\!\!$*}
\address{Department of Mathematics, Sun Yat-Sen University, 
Guangzhou 510275, China}
\email{mcsjzl@mail.sysu.edu.cn}
\thanks{*Corresponding author. {\it Email:} mcsjzl@mail.sysu.edu.cn({Z.} Jiang)}


\date{\today}



\begin{abstract}
Let $u=u(t,{\bf x},{\bf p})$ satisfy the transport equation $\frac {\partial u}{\partial t}+\frac {{\bf p}}{p_0}\frac{\partial u}{\partial{\bf x}}=f$, 
where $f=f(t,\bf x,\bf p)$ belongs to $ L^{p}((0,T)\times {\bf R}^{3}\times {\bf R}^{3})$ for $1<p<\infty$ and 
$\frac {\partial}{\partial t}+\frac {{\bf p}}{p_0}\frac{\partial}{\partial{\bf x}}$ is the relativistic-free transport operator. 
We show the regularity of $\int_{{\bf R}^{3}}u(t, {\bf x}, {\bf p})d{\bf p}$ using the same method as given by Golse, Lions, Perthame and Sentis. 
This average regularity is considered in terms of fractional Sobolev spaces 
and it is very useful for the study of the existence of the solution to the Cauchy problem on the relativistic Boltzmann equation.
\end{abstract}

\maketitle
\vspace*{-0.7cm}
\section{Introduction}
\label{intro}
We are concerned with the average regularity of the solution to an equation 
with the relativistic-free transport operator from the relativistic Boltzmann equation.
Let us begin with the relativistic Boltzmann equation in the following form:
\begin{equation}
\frac {\partial u}{\partial t}+\frac {{\bf p}}{p_0}\frac{\partial u}{\partial{\bf x}}=Q(u,u),
\end{equation}

\noindent where $u=u(t,{\bf x},{\bf p})$ is a distribution function of a one-particle relativistic gas 
with the time $t\in(0,\infty)$, the position ${\bf x}\in {\bf R}^{3}$, and the momentum ${\bf p}\in {\bf R}^{3}$; 
 $p_0=(1+|{\bf p}|^2)^{1/2}$ denotes the energy of a dimensionless relativistic gas particle with the momentum ${\bf p}$;
$\frac {\partial }{\partial t}+\frac {{\bf p}}{p_0}\frac{\partial }{\partial{\bf x}}$ is called the relativistic-free transport operator; 
$Q(u,u)$ is the relativistic Boltzmann collision operator  which can be written as the difference between the gain and loss terms respectively 
given by Dudy\'{n}ski and Ekiel-Je\.{z}ewska \cite{de88} in the following forms:
\begin{equation}
\label{qplus}
Q^{+}(u, u)=\int_{{\bf R}^{3}\times S^{2}} \frac{gs^{1/2}}{p_{0}p_{\ast 0}}\sigma(g, \theta) u( t, {\bf x}, {\bf p}^{\prime})u( t, {\bf x}, {\bf p}_{\ast}^{\prime})
d{\omega}d{\bf p}_{\ast},
\end{equation}
\begin{equation}
\label{qminus}
Q^{-}(u, u)=\int_{{\bf R}^{3}\times S^{2}} \frac{gs^{1/2}}{p_{0}p_{\ast 0}}\sigma(g, \theta) u( t, {\bf x}, {\bf p})u( t, {\bf x}, {\bf p}_{\ast})
d{\omega}d{\bf p}_{\ast}.
\end{equation}
It is worth mentioning that the gain and loss terms of the relativistic Boltzmann equation can be expressed in other various forms (see \cite{gv}).  
The other different parts in {eqs.} (\ref{qplus}) and (\ref{qminus}) are explained as follows.

${\bf p}$ and $ {\bf p}_{\ast}$ are dimensionless momenta of two relativistic  particles 
immediately before collision while ${\bf p}^{\prime}$ and $ {\bf p}_{\ast}^{\prime}$ are dimensionless momenta after collision; 
$p_{\ast 0}=(1+|{\bf p}_{\ast}|^2)^{1/2}$ denotes the dimensionless energy of the colliding relativistic gas particle 
with the momentum ${\bf p}_{\ast}$ before collision, and as used below in the same way, 
$p_{0}^{\prime}=(1+|{\bf p}^{\prime}|^2)^{1/2}$ and $p_{\ast 0}^{\prime}=(1+|{\bf p}_{\ast}^{\prime}|^2)^{1/2}$  
represent the dimensionless energy of the two relativistic  particles after collision.
$s={|p_{\ast0}+p_{0}|^{2}-|{\bf p}_{\ast}+{\bf p}|^{2}}$  and $s^{1/2}$ is the total energy in the center-of-mass frame \cite{de88}; 
$g=\sqrt{|{\bf p}_{\ast}-{\bf p}|^{2}-|p_{\ast0}-p_{0}|^{2}}/2$ and $2g$ is in fact the value of the relative momentum 
in the center-of-mass frame \cite{de88}; it can be seen that $s=4+4g^{2}$. 
$\sigma(g, \theta)$ is the differential scattering cross section of  the variable $g$ and the scattering angle $\theta$. 
${\bf R}^{3}$ is a three-dimensional Euclidean space and $S^2$ a unit sphere surface with an infinitesimal element 
$d{\omega}=\sin\theta d\theta d\varphi$ for the scattering angle $\theta\in [0,\pi]$ and the other solid angle $\varphi\in [0,2\pi]$ 
in the center-of-momentum system, and the scattering angle $\theta$ is defined by  
$\cos\theta=1-2[(p_{0}-p_{\ast0})(p_{0}-p_{0}^{\prime})-({\bf p}-{\bf p}_{\ast})({\bf p}-{\bf p}^{\prime})]/(4-s)$.

There is a long history of the study of the relativistic Boltzmann equation as one of the most important in the relativistic kinetic theory. 
The study of the relativistic kinetic theory began in 1911 when J$\ddot{u}$ttner \cite{ju} derived an equilibrium distribution function of relativitic gases.
Lichnerowicz and Marrot \cite{lm} were the first to derive the full relativistic Boltzmann equation including the collision operator in 1940.
The research of this eqution can be roughly classified into four aspects: 
1) the derivation of this equation; 
2) its relativistic hydrodynamic limit. 
3) its Chapman-Enskog approximation and hydrodynamic modes;
4) the existence and uniqueness of the solution to the Cauchy problem on it;
For  both of 1) and 2), 
we can see the recent references from Dolan \cite{pz}, Debbasch \& Leeuwen \cite{fl1} \cite{fl2}, Tsumura \& Kunihiro \cite{tk} and Denicol et {al.}~\cite{dhm}. 
For 3), in the early 60's, many researchers, such as Israel \cite{i}, applied the Chapman-Enskog expansion 
into studying the aproximative solution to the relativistic Boltzmann equation. 
The progress of the last research fields are recently great. 
In 1967, Bichteler \cite{b} first proved  that the relativistic Boltzmann equation admits a unique local solution 
under the assumptions that the differential scattering cross-section is bounded and that the intial distribution function decays exponentially with energy.
In 1988, Dudy\'{n}ski and Ekiel-Je\.{z}ewska \cite{de88} proved that the Cauchy problem  
on the lineared relativistic Boltzmann equation has a unique solution in $L^{2}$ space.
Four years later, Dudy\'{n}ski and Ekiel-Je\.{z}ewska \cite{de92} showed that there exists a DiPerna-Lions renormalized solution \cite{dl} 
to the Cauchy problem on the relativistic Boltzmann equation with large initial data in the case of the relativistic soft interactions. 
Glassey and Strauss \cite{gs93} proved in 1993 that a unique global smooth solution 
to this problem exponentially converges to a relativistic Maxwellian as the time goes to infinity 
if all initial data are periodic in the space variable and near equilibrium. 
Then in 1995, Glassey and Strauss extended the above result to the whole space case \cite{gs95} and found that the solution 
has the property of polynomial convergence with respect to the time.
 In 1996, Andr\'{e}asson \cite{a96} showed the regularity of the gain term and the strong $L^{1}$ convergence 
to equilibrium for the relativistic Boltzmann equation.
Afterward, Jiang gave the global existence of solution to the relativistic Boltzmann equation with hard interactions 
in the whole space for initial data with finite mass, energy and inertia \cite{j}, or in a periodic box 
for initial data with finite mass and energy \cite{j98} \cite{j99}.
In 2004, Andr\'{e}asson, Calogero and Illner \cite{aci04} obtained the property that the solution to the gain-term-only relativistic Boltzmann equation 
blows up in finite time.
In 2006, Glassey \cite{g06} showed a unique global solution to the relativistic Boltzmann equation 
with initial data  near vacuum state for certain differential scattering cross section.
In 2008, Jiang \cite{j07} obtained the global existence of solution to the relativistic Boltzmann equation with hard interactions in the whole space 
for initial data with finite mass and energy.
Recently, Strain gave a Newtonian limit of global solutions to the Cauchy problem on 
the relativistic  Boltzmann equation near vacuum \cite{s2010a}, showed the asymptotic stability of the relativistic Boltzmann equantion 
for the soft potentials \cite{s2010b} and made a survey of various equivalent expressions of the relativistic Boltzmann equantion \cite{s2011}. 
There are also many  authors who are contributed to the study of relativistic Boltzmann equation, e.g., 
Escobedo et {al.}~\cite{emv}, Ha et {al.}~\cite{hk},  Hsiao \& Yu \cite{hy}, Swart \cite{s}. 
Many other relevant results can be found in the references mentioned above.

In this paper, we show the regularity of the momentum average of the solution to an equation 
with the relativistic-free transport operator 
by use of the same method as given by Golse, Lions, Perthame and Sentis \cite{gl}.
This regularity is very useful for our further study of the existence and uniqueness of the solution to the Cauchy problem 
on the relativistic Boltzmann equation. It is presented by the following theorem: 
\begin{theorem}
\label{th1}
Assume that $u$ and $f$ satisfy the following equation:
\begin{equation}
\label{traneqn}
\frac {\partial u}{\partial t}+\frac {{\bf p}}
{p_0}\frac{\partial u}{\partial{\bf x}}=f 
\end {equation}
 where ${\sf supp} u\subset [\varepsilon_{0},T-\varepsilon_{0}]\times {\bf R}^{3}\times B_{R}$, $T\in(0,\infty), \varepsilon_{0}\in(0,\frac{T}{2})$ 
and $B_R$ is a ball centered on the origin and with a radius of $R$.
 If $u$ and $f$ 
both belong to $L^{p}((0,T)\times {\bf R}^{3}\times {\bf R}^{3})$ for $p \in (1, +\infty)$,
 we can conclude that $\widetilde{u}=\int_{{\bf R}^{3}}u(t, {\bf x}, {\bf p})d{\bf p}\in W^{s,p}((0,T)\times {\bf R}^{3})$  
for every $s$ satisfying $0<s<\inf(\frac{1}{p},1-\frac{1}{p})$, and what's more, there exists a constant $C$ 
which does not depend on $u$ and $f$, such that 
\begin{equation}
\label{result}
\|\widetilde{u}\|_{W^{s,p}}\leq C \| u \|_{L^{p}}^{1-s} \|f\|_{L^{p}}^{s}
\end{equation}
\noindent where 
\begin{equation}
\label{norm}
\|\widetilde{u}\|_{W^{s,p}}=(\int_{0}^T\int_{{\bf R}^{3}}\int_{0}^T\int_{{\bf R}^{3}}
\frac{|\widetilde{u(t_{1},{\bf x}_{1})}-\widetilde{u(t_{2},{\bf x}_{2})}|^{2}}{[(t_{1}-t_{2})^{2}
+({\bf x}_{1}-{\bf x}_{2})^{2}]^{\frac{4+sp}{2}}}d{\bf x}_{1}dt_{1}d{\bf x}_{2}dt_{2})^{1/p}.
\end{equation}
\end{theorem}

The norm of $\widetilde{u}$, given in (\ref{norm}), is defined in a 
fractional Sobolev space which can be seen in the book of  Triebel \cite{t}. 
In order to prove Theorem \ref{th1}, we first have to consider the special case when $p=2$, that is, 
we investigate the regularity of $\int_{{\bf R}^{3}}u(t, {\bf x}, {\bf p})d{\bf p}$ 
when $u$ and $f$ belong to the $L^{2}$ space. 
This regularity is in fact presented by the following lemma: 

\begin{lemma}[\cite{j97}]
\label{le2}
Assume that $u$ and $f$ satisfy (\ref{traneqn}) and that 
 ${\sf supp} u\subset [\varepsilon_{0},T-\varepsilon_{0}]\times {\bf R}^{3}\times B_{R}$, where 
$T\in(0,\infty), \varepsilon_{0}\in(0,\frac{T}{2})$ and $B_R$ is a ball centered on the origin and with a radius of $R$.
 If $u$ and  $f$ belong to $L^{2}((0,T)\times {\bf R}^{3}\times {\bf R}^{3})$,
 we can conclude that $\widetilde{u}=\int_{{\bf R}^{3}}u(t, {\bf x}, {\bf p})d{\bf p}\in W^{\frac{1}{2},2}((0,T)\times {\bf R}^{3})$, 
and what's more, there exists a constant $C$ 
which does not depend on $u$ and $f$, such that
\begin{equation}
\|\widetilde{u}\|_{H^{\frac{1}{2}}}\leq C\| u \|_{L^{2}}^{1/2} \|f\|_{L^{2}}^{1/2}
\end{equation}
\noindent where 
\begin{equation}
\label{normh1}
\|\widetilde{u}\|_{H^{\frac{1}{2}}((0,T)\times {\bf R}^{3})}=
(\int_{0}^T\int_{{\bf R}^{3}}\int_{0}^T\int_{{\bf R}^{3}}\frac{|\widetilde{u(t_{1},{\bf x}_{1})}-\widetilde{u(t_{2},{\bf x}_{2})}|^{2}}
{[(t_{1}-t_{2})^{2}+({\bf x}_{1}-{\bf x}_{2})^{2}]^{\frac{5}{2}}}d{\bf x}_{1}dt_{1}d{\bf x}_{2}dt_{2})^{1/2}.
\end{equation}
\end{lemma}

Lemma \ref{le2} was given by Jiang \cite{j97} (or see Appendix \ref{app}) in 1997.  
His proof is similiar to that used by Golse, Lions, Perthame and Sentis \cite{gl}, that is, 
it is mainly based on the analysis of the integral in the norm $\|\widetilde{u}\|_{H^{\frac{1}{2}}({\bf R}\times {\bf R}^{3})}$ 
 by use of an estimate of a special integral in a subset 
in the ball $B_{R}=\{{\bf p}: |{\bf p}|<R, {\bf p}\in {\bf R}^3\}$.

It is worth mentioning that Lemma \ref{le2} can be regarded as a special case of the well established result given by Rein \cite{r2004} in 2004. 
Rein's result includes the fact that  
if $u$ has compact supports in $L^{2}({\bf R} \times {\bf R}^{3}\times B_{R})$ and satisfies
\begin{equation}
\label{traneqn1}
\frac {\partial u}{\partial t}+\frac {{\bf p}}
{p_0}\frac{\partial u}{\partial{\bf x}}=g_0+{\rm div}_{\bf p}g_{1}
\end {equation}
for $g_0$ and $g_1$  in $L^{2}({\bf R} \times {\bf R}^{3}\times B_{R})$, 
then the momentum average $\widetilde{u}$ belongs to $H^{1/4}({\bf R} \times {\bf R}^{3})$. 
When $g_{1}$ is equal to zero,
(\ref{traneqn1}) is reduced to be (\ref{traneqn}).
Jiang's result in Lemma \ref{le2} shows that in this case $\widetilde{u}$ belongs to $H^{1/2}({\bf R} \times {\bf R}^{3})$.  
Since $u\in L^{2}({\bf R}\times {\bf R}^{3}\times B_{R})$,  
it can be further found that $\widetilde{u}$ belongs to $L^{2}({\bf R} \times  {\bf R}^{3})$ as well. 
Hence, by use of Proposition 1.32 in \cite{bcd}, 
it can be known that $\widetilde{u}$ belongs to $H^{1/4}({\bf R} \times {\bf R}^{3})$ in Lemma \ref{le2} too. 

Theorem \ref{th1} is a generalization of Lemma  \ref{le2} into the $L^{p}$ space for $p\in (1,+\infty)$ and it can be 
 proved by the analysis of the operator $\mathfrak{F}: (u,f)\rightarrow \widetilde{u}$,  
with the help of the complex interpolation method (see \cite{bl}\cite{t}).
The detailed proof of Theorem \ref{th1} will be shown in the next section.

\section{Proof for Theorem \ref{th1}}
\label{sec2}
In order to prove Theorem \ref{th1}, we first observe some properties of the solutions to  {Eqn.}~(\ref{traneqn}).
Notice that  $u$ and $f$ satisfy {Eqn.}~(\ref{traneqn}), 
and that ${\sf supp} u\subset [\varepsilon_{0},T-\varepsilon_{0}]\times {\bf R}^{3}\times B_{R}$.  
It follows that $u$ is in fact a unique solution to  the Cauchy problem on {Eqn.}~(\ref{traneqn}) with its initial datum zero 
when $f\in L^{p}((0,T)\times {\bf R}^{3}\times {\bf R}^{3})$ for any $p\in (1,+\infty)$. 
Denote $u^{\sharp}(t, {\bf x}, {\bf p})\triangleq u(t, {\bf x}+\frac{\bf p}{p_{0}}t, {\bf p})$ 
and $f^{\sharp}(t, {\bf x}, {\bf p})\triangleq f(t, {\bf x}+\frac{\bf p}{p_{0}}t, {\bf p})$.   
Then {Eqn.}~(\ref{traneqn}) can be rewritten as 
\begin{equation}
\label{3.2}
\frac {\partial u^{\sharp}}{\partial t}+\frac {{\bf p}}
{p_0}\frac{\partial u^{\sharp}}{\partial{\bf x}}=f^{\sharp}.
\end {equation}
(\ref{3.2})  is equivalent to the following equation:
\begin{equation}
\label{3.3}
\frac {d u^{\sharp}}{d t}=f^{\sharp}. 
\end {equation}
Hence
\begin{equation}
\label{3.4}
u^{\sharp}(t,{\bf x},{\bf p})=\int^{t}_{0}f^{\sharp}(s, {\bf x}, {\bf p})ds,
\end{equation}
that is,
\begin{equation}
\label{3.4}
u(t,{\bf x},{\bf p})=\int^{t}_{0}f(s, {\bf x}+\frac{\bf p}{p_{0}}(s-t), {\bf p})ds.
\end{equation}

Let $h(t,{\bf x}, {\bf p})\triangleq u(t,{\bf x}, {\bf p})+f(t,{\bf x}, {\bf p})$.
 Then {Eqn.}~(\ref{traneqn}) can be rechanged as
\begin{equation}
\label{h1}
u+\frac {\partial u}{\partial t}+\frac {{\bf p}}
{p_0}\frac{\partial u}{\partial{\bf x}}=h,
\end {equation}
 thus $u(t, {\bf x}, {\bf p})$ can be expressed as 
\begin{equation}
\label{h2}
u(t,{\bf x},{\bf p})=\int^{t}_{0}e^{s-t}h(s, {\bf x}+\frac{\bf p}{p_{0}}(s-t), {\bf p})ds,
\end{equation}
and it is still a unique solution to the Cauchy problem on {Eqn.}~(\ref{h1}) with its initial datum zero 
when $f\in L^{p}((0,T)\times {\bf R}^{3}\times {\bf R}^{3})$ for any $p\in (1,+\infty)$. 
Notice that $h(t, {\bf x}, {\bf p})$ has compact support 
since ${\sf supp} u\in [\varepsilon_{0},T-\varepsilon_{0}]\times {\bf R}^{3}\times B_{R}$
and that $h$ belong to $L^{q}((0,T)\times {\bf R}^{3}\times {\bf R}^{3})$ if $u$ and $f$ both 
belong to $L^{q}((0,T)\times {\bf R}^{3}\times {\bf R}^{3})$ for any $q\in [1, +\infty]$.
Then we can define a linear operator $\mathfrak{F}$ 
from $L^{q}_{R}((0,T)\times {\bf R}^{3}\times {\bf R}^{3})$ into $L^{q}((0,T)\times {\bf R}^{3})$ as follows:
%
\begin{equation}
\label{oper2}
 \mathfrak{F}(h)=\widetilde{u},
\end{equation}
for every $q\in [1, +\infty]$, 
where   
 $L^{q}_{R}((0,T)\times {\bf R}^{3}\times {\bf R}^{3})$ is defined by 
\begin{equation}
\label{lqr}
L^{q}_{R}=\{ h\in L^{q}((0,T)\times {\bf R}^{3}\times {\bf R}^{3}): h \hbox{ has compact support in } (0,T)\times {\bf R}^{3}\times B_{R} \} 
\end{equation}
and $\widetilde{u}(t,{\bf x})=\int_{{\bf R}^{3}}u(t,{\bf x}, {\bf p})d{\bf p}$ for any $u(t,{\bf x}, {\bf p})$ in (\ref{h2}). 

We below prove that $\mathfrak{F}$ is 
bounded from $L^{q}_{R}((0,T)\times {\bf R}^{3}\times {\bf R}^{3})$ into $L^{q}((0,T)\times {\bf R}^{3})$ for every $q\in [1,+\infty]$.

By H\"{o}lder's inequality, we obtain the following estimate: 
\begin{eqnarray*}
\label{lq1}
\|\widetilde{u}\|_{L^{q}((0,T)\times {\bf R}^{3})}&=&(\int_{0}^{T}\int_{{\bf R}^{3}}|\widetilde{u}(t,{\bf x})|^{q}d{\bf x}dt)^{\frac{1}{q}}\\
 &=&(\int_{0}^{T}\int_{ {\bf R}^{3}}|\int_{B_{R}}u(t,{\bf x},{\bf p})d{\bf p}|^{q}d{\bf x}dt)^{\frac{1}{q}}\\
 &\leq&(\int_{0}^{T}\int_{ {\bf R}^{3}}|(\int_{B_{R}}1^{\frac{q}{q-1}}d{\bf p})^{1-\frac{1}{q}}(\int_{B_{R}}|u(t,{\bf x},{\bf p})|^{q}d{\bf p})^{\frac{1}{q}}|^{q}dtd{\bf x})^{\frac{1}{q}}\\
 &=&\int_{0}^{T}\int_{ {\bf R}^{3}}|(\frac{4}{3}\pi {R}^{3})^{1-\frac{1}{q}}(\int_{B_{R}}|u(t,{\bf x},{\bf p})|^{q}d{\bf p})^{\frac{1}{q}}|^{q}dtd{\bf x})^{\frac{1}{q}}
\end{eqnarray*}
\begin{eqnarray}
\hspace*{0.5cm} =& (\frac{4}{3}\pi {R}^{3})^{1-\frac{1}{q}}(\int_{0}^{T}\int_{ {\bf R}^{3}}\int_{B_{R}}|u(t,{\bf x},{\bf p})|^{q}d{\bf p}d{\bf x}dt)^{\frac{1}{q}}
\end{eqnarray}
for $q \in (1,+\infty)$.
\noindent Because of (\ref{h2}), using H\"{o}lder's inequality again, we have
\begin{eqnarray*}
\label{lq2}
&&(\int_{0}^{T}\int_{ {\bf R}^{3}}\int_{B_{R}}|u(t, {\bf x}, {\bf p})|^{q}d{\bf p}d{\bf x}dt)^{\frac{1}{q}}\\
&=& (\int_{0}^{T}\int_{{\bf R}^{3}}\int_{B_{R}}|\int^{t}_{0}e^{s-t}h(s, {\bf x}+\frac{\bf p}{p_{0}}(s-t), {\bf p})ds|^{q}d{\bf p}d{\bf x}dt)^{\frac{1}{q}}\\
&\leq&(\int_{0}^{T}\int_{{\bf R}^{3}}\int_{B_{R}}|(\int_{0}^{t}e^{\frac{q(s-t)}{q-1}}ds )^{1-1/q}(\int_{0}^{t} |h(s, {\bf x}+\frac{\bf p}{p_{0}}(s-t), {\bf p})|^{q}ds)^{1/q}|^{q}d{\bf p}d{\bf x}dt)^{\frac{1}{q}}
\end{eqnarray*}
\begin{eqnarray}
\leq C_1(\int_{0}^{T}\int_{{\bf R}^{3}}\int_{B_{R}}\int_{0}^{t} |h(s, {\bf x}+\frac{\bf p}{p_{0}}(s-t), {\bf p})|^{q}dsd{\bf p}d{\bf x}dt)^{\frac{1}{q}},
\end{eqnarray}
where $C_1=\left(\frac{q-1}{q}\right)^\frac{q-1}{q}(1-e^{-T})$.
Let ${\bf y}={\bf x}+\frac{{\bf p}}{p_{0}}(s-t)$. 
By Fubini's theorem, we can first calculate the last integral in (\ref{lq2}) with respect to  ${\bf x}$, 
thus, by transforming ${\bf x}$ into ${\bf y}$, giving
\begin{eqnarray*}
\label{lq3}
 &&(\int_{0}^{T}\int_{{\bf R}^{3}}\int_{B_{R}}\int_{0}^{t} |h(s, {\bf x}+\frac{\bf p}{p_{0}}(s-t), {\bf p})|^{q}dsd{\bf p}d{\bf x}dt)^{\frac{1}{q}}
\end{eqnarray*}
\begin{eqnarray}
=(\int_{0}^{T}\int_{{\bf R}^{3}}\int_{B_{R}}\int_{0}^{t} |h( s, {\bf y}, {\bf p})|^{q}d{s}d{\bf p}d{\bf y}d{t})^{\frac{1}{q}}. 
\end{eqnarray}
 It follows that
\begin{eqnarray}
\label{lq4}
 &&(\int_{0}^{T}\int_{{\bf R}^{3}}\int_{B_{R}}\int_{0}^{t} |h( s, {\bf y}, {\bf p})|^{q}d{s}d{\bf p}d{\bf y}d{t})^{\frac{1}{q}}\nonumber \\
 &\leq& (T)^{1/q}(\int_{{\bf R}^{3}}\int_{B_{R}}\int_{0}^{T} |h( s, {\bf y}, {\bf p})|^{q}d{s}d{\bf p}d{\bf y})^{\frac{1}{q}} \nonumber \\
 &=& (T)^{1/q} \| h \|_{L^{q}}.
\end{eqnarray}
By (\ref{lq1}), (\ref{lq2}), (\ref{lq3}) and (\ref{lq4}), we conclude that 
\begin{eqnarray}
\label{lq5}
\|\widetilde{u}\|_{L^{q}} \leq C_2\| h \|_{L^{q}} ,
\end{eqnarray}
where $C_2=\left(\frac{4\pi R^3}{3}\right)^{1-\frac{1}{q}}T^{1/q}\left(\frac{q-1}{q}\right)^\frac{q-1}{q}$.
 Similarly, we can consider two spectial cases: one is for $p=1$ and another is for $p=+\infty$.  
In the $p=1$ case, for  any $h \in L^{1}_{R}((0,T)\times {\bf R}^{3}\times {\bf R}^{3})$, we have
\begin{eqnarray*}
\label{l11}
\|\widetilde{u}\|_{L^{1}((0,T)\times {\bf R}^{3})}&=&\int_{0}^{T}\int_{ {\bf R}^{3}}|\widetilde{u}(t,{\bf x})|d{\bf x}dt\\
 &=&\int_{0}^{T}\int_{ {\bf R}^{3}}|\int_{B_{R}}u(t,{\bf x},{\bf p})d{\bf p}|d{\bf x}dt\\
 &\leq&\int_{0}^{T}\int_{ {\bf R}^{3}\times B_{R}}|u(t,{\bf x},{\bf p})|d{\bf p}d{\bf x}dt\\
 &=& \int_{0}^{T}\int_{{\bf R}^{3}}\int_{B_{R}}|\int^{t}_{0}e^{s-t}h(s, {\bf x}+\frac{\bf p}{p_{0}}(s-t), {\bf p})ds|d{\bf p}d{\bf x}dt\\
 &\leq& \int_{0}^{T}\int_{{\bf R}^{3}}\int_{B_{R}}\int_{0}^{t}|e^{s-t}h(s, {\bf x}+\frac{\bf p}{p_{0}}(s-t), {\bf p})|dsd{\bf p}d{\bf x}dt
\end{eqnarray*}
\begin{eqnarray}
 \hspace*{3cm}\leq\int_{0}^{T}\int_{{\bf R}^{3}}\int_{B_{R}}\int_{0}^{t}|h(s, {\bf x}+\frac{\bf p}{p_{0}}(s-t), {\bf p})|dsd{\bf p}d{\bf x}dt.
\end{eqnarray}
By Fubini's theorem, transforming ${\bf x}$ into ${\bf y}={\bf x}+\frac{\bf p}{p_{0}}(s-t)$, we can get
\begin{eqnarray*}
\label{l12}
 &&\int_{0}^{T}\int_{{\bf R}^{3}}\int_{B_{R}}\int_{0}^{t}| h(s, {\bf x}+\frac{\bf p}{p_{0}}(s-t), {\bf p})|dsd{\bf p}d{\bf x}dt\\
 &=& \int_{0}^{T}\int_{{\bf R}^{3}}\int_{B_{R}}\int_{0}^{t}|h(s, {\bf y}, {\bf p})|d{s}d{\bf p}d{\bf y}d{t}\\
 &\leq&T\int_{{\bf R}^{3}}\int_{B_{R}}\int_{0}^{t}|h( s, {\bf y}, {\bf p})|dsd{\bf p}d{\bf y}
\end{eqnarray*}
\begin{eqnarray}
\leq C_3\|h\|_{L^{1}},
\end{eqnarray}
where $C_3=T$. 
(\ref{l11}) and (\ref{l12}) yield 
\begin{eqnarray}
\label{l13}
\|\widetilde{u}\|_{L^{1}}\leq C_3\|h\|_{L^{1}}.
\end{eqnarray}

In the $p=+\infty$ case,  for $h \in L^{\infty}_{R}((0.T)\times {\bf R}^{3}\times {\bf R}^{3})$, we can also get 
\begin{eqnarray*}
\label{linf1}
\|\widetilde{u}\|_{L^{\infty}((0,T)\times {\bf R}^{3})}&=&\sup\limits_{(t,{\bf x})\in(0,T)\times {\bf R}^{3}}|\widetilde{u}(t,{\bf x})|\\
 &=&\sup\limits_{(t,{\bf x})\in(0,T)\times {\bf R}^{3}}|\int_{B_{R}}u(t,{\bf x}, {\bf p})d{\bf p}|\\
 &=&\sup\limits_{(t,{\bf x})\in(0,T)\times {\bf R}^{3}}|\int_{B_{R}}\int^{t}_{0}e^{s-t}h(s, {\bf x}+\frac{\bf p}{p_{0}}(s-t), {\bf p})dsd{\bf p}|\\
 &\leq&\sup\limits_{(t,{\bf x})\in(0,T)\times {\bf R}^{3}}\int_{B_{R}}\int_{0}^{t}|e^{s-t}h(s, {\bf x}+\frac{\bf p}{p_{0}}(s-t), {\bf p})|dsd{\bf p}
\end{eqnarray*}
\begin{eqnarray}
 \leq \sup\limits_{(t,{\bf x})\in(0,T)\times {\bf R}^{3}}\int_{B_{R}}\int_{0}^{t}|h(s, {\bf x}+\frac{\bf p}{p_{0}}(s-t), {\bf p})|dsd{\bf p}.
\end{eqnarray}
By Fubini's theorem, transforming ${\bf x}$ into ${\bf y}=\bf{x}+\frac{\bf p}{p_{0}}(s-t)$,  we have
\begin{eqnarray*}
\label{linf2}
  &&\sup\limits_{(t,{\bf x})\in(0,T)\times {\bf R}^{3}}\int_{B_{R}}\int_{0}^{t}|h(s, {\bf x}+\frac{\bf p}{p_{0}}(s-t), {\bf p})|dsd{\bf p}\\
  &=&\sup\limits_{(t,{{\bf y}})\in(0,T)\times {\bf R}^{3}}\int_{B_{R}}\int_{0}^{t}|h(s,{\bf y},{\bf p})|d{s}d{\bf p}\\
 &\leq&\sup\limits_{(t,{\bf y})\in(0,T)\times {\bf R}^{3}}(\frac{4}{3}\pi {R}^{3}{t})\sup\limits_{(s,{\bf p})\in(0,t)\times B_{R}}|h(s,{\bf y},{\bf p})|\\
 &\leq& (\frac{4}{3}\pi {R}^{3}T) \sup\limits_{(s, {\bf y}, {\bf p})\in(0,T)\times {\bf R}^{3}\times B_{R}} |h(s,{\bf y},{\bf p})|
\end{eqnarray*}
\begin{eqnarray}
&\leq& C_4\| h \|_{L^{\infty}},
\end{eqnarray}
where $C_4=\frac{4}{3}\pi {R}^{3}T$. 
By (\ref{linf1}) and (\ref{linf2}), we get
\begin{eqnarray}
\label{linf3}
\|\widetilde{u}\|_{L^{\infty}}\leq C_4\|h\|_{L^{\infty}}.
\end{eqnarray}

Hence  $\mathfrak{F}$ is bounded from $L^{q}_{R}((0,T)\times {\bf R}^{3}\times {\bf R}^{3})$ to $L^{q}((0,T)\times {\bf R}^{3})$ for every $1\leq q\leq\infty$.

By using Lemma \ref{le2}, we below show  that  $\mathfrak{F}$ is bounded 
from $L^{2}_{R}((0,T)\times {\bf R}^{3}\times {\bf R}^{3})$  into $H^{\frac{1}{2}}((0,T)\times {\bf R}^{3})$.
Setting $\alpha=1$ in  (\ref{2}) and using Plancherel's identity,  we have
\begin{eqnarray*}
\label{h2l2}
\|\widetilde{u}\|_{H^{\frac{1}{2}}(R\times {\bf R}^{3})}^{2}&=&\int_{R\times {\bf R}^{3}}|{\bf z}^{2}+\tau^{2}|^{1/2} |\int_{{\bf R}^{3}}\widehat{u}(\tau, {\bf z},{\bf p})d{\bf p}|^{2}d{\bf z}d{\tau}\\
&\leq&  2C_R( \int_{R\times {\bf R}^{3}\times {\bf R}^{3}}|\widehat{u}|^{2}d{\bf p}d{\bf z}d{\tau}+\int_{R\times {\bf R}^{3}\times {\bf R}^{3}}|\frac{{\bf p}\cdot\bf{z}}{\sqrt{1+{\bf p}^{2}}}+\tau|^{2}|\widehat{u}|^{2}d{\bf p}d{\bf z}d{\tau})\\
&=&      2C_R( \int_{R\times {\bf R}^{3}\times {\bf R}^{3}}|u|^{2}d{\bf p}d{\bf x}d{t}+\int_{R\times {\bf R}^{3}\times {\bf R}^{3}}|\frac {\partial u}{\partial t}+\frac {{\bf p}}{p_0}\frac{\partial u}{\partial{\bf x}}|^{2}d{\bf p}d{\bf x}d{t})\\
&=&      2C_R( \int_{R\times {\bf R}^{3}\times {\bf R}^{3}}|u|^{2}d{\bf p}d{\bf x}d{t}+\int_{R\times {\bf R}^{3}\times {\bf R}^{3}}| f |^{2}d{\bf p}d{\bf x}d{t})\\
&=&      2C_R( \int_{R\times {\bf R}^{3}\times {\bf R}^{3}}|u|^{2}d{\bf p}d{\bf x}d{t}+\int_{R\times {\bf R}^{3}\times {\bf R}^{3}}| h-u |^{2}d{\bf p}d{\bf x}d{t})
\end{eqnarray*}
\begin{eqnarray}
\leq      6C_R( \int_{R\times {\bf R}^{3}\times {\bf R}^{3}}|u|^{2}d{\bf p}d{\bf x}d{t}+\int_{R\times {\bf R}^{3}\times {\bf R}^{3}}| h |^{2}d{\bf p}d{\bf x}d{t}),
\end{eqnarray}
where $C_R$ is the same as in Lemma \ref{le2}. 
Because $u(t,{\bf x},{\bf p})=\int^{t}_{0}e^{s-t}h(s, {\bf x}+\frac{\bf p}{p_{0}}(s-t), {\bf p})ds$, 
by using (\ref{lq2}), (\ref{lq3}) and (\ref{lq4}) with $q=2$, (\ref{h2l2}) yields
\begin{eqnarray}
\|\widetilde{u}\|_{H^{\frac{1}{2}}}\leq C_5\| h\| _{L^{2}},
\end{eqnarray}
where $C_5=\sqrt{6}(1+T/2)^{1/2}C_R^{1/2}$.
Hence  the operator $\mathfrak{F}$ is bounded from $L^{2}_{R}((0,T)\times {\bf R}^{3}\times {\bf R}^{3})$  into $H^{\frac{1}{2}}((0,T)\times {\bf R}^{3})$.

 By using complex interpolation method (\cite{bl} \cite{t}),  we can know that  $\mathfrak{F}$ is bounded from 
$(L^{q}_{R},L^{2}_{R})_{[\theta]}((0,T)\times {\bf R}^{3}\times {\bf R}^{3})$ into $(L^{q},H^{\frac{1}{2}})_{[\theta]}((0,T)\times {\bf R}^{3})$.
Here, for $1<q<+\infty$, 
\begin{equation}
\label{compin1}
(L^{q}_{R},L^{2}_{R})_{[\theta]}=L^{p}_{R}  \hbox{ for }  \frac{1}{p}=\frac{1-\theta}{q}+\frac{\theta}{2}, \ 0\leq\theta\leq 1
\end{equation}
and
\begin{equation}
\label{compin}
(L^{q},H^{\frac{1}{2}})_{[\theta]}=W^{s,p}   \hbox{ for }   \frac{1}{p}=\frac{1-\theta}{q}+\frac{\theta}{2} ,\ s=\frac{\theta}{2},\  0\leq\theta\leq 1.
\end{equation}
In particular, for $q=1$,
\begin{equation}
\frac{1}{p}=1-\frac{\theta}{2},\ s=\frac{\theta}{2}=1-\frac{1}{p};
\end{equation}
for $q=+\infty$,
\begin{equation}
\frac{1}{p}=\frac{\theta}{2},\ s=\frac{\theta}{2}=\frac{1}{p}.
\end{equation}
Hence it is necessary  to  set the following condition in Theorem \ref{th1}:
\begin{equation}
0<s<\inf(1-\frac{1}{p},\frac{1}{p}).
\end{equation}

According to the above analysis, we can conclude that $\mathfrak{F}$ is a bounded linear operator 
from $L^{p}_{R}((0,T)\times {\bf R}^{3}\times {\bf R}^{3})$ into $W^{s,p}((0,T)\times {\bf R}^{3})$ 
for any $1<p<\infty$ and $0<s<\inf(1-\frac{1}{p},\frac{1}{p})$. It follows that 
  $\widetilde{u}\in W^{s,p}((0,T)\times {\bf R}^{3})$ if  $u(t,{\bf x},{\bf p})$ and 
$f(t,{\bf x},{\bf p})$ both belong to $ L^{p}_{R}((0,T)\times {\bf R}^{3}\times {\bf R}^{3})$  for any $ 1<p<\infty$. 
Since ${\sf supp} u\subset [\varepsilon_{0},T-\varepsilon_{0}]\times {\bf R}^{3}\times B_{R}$, 
we can extend the domain of $u$ from $(0,T)\times {\bf R}^{3}\times {\bf R}^{3}$ 
into ${\bf R}\times {\bf R}^{3}\times {\bf R}^{3}$. Moreover, using Minkowski's inequality, we have
\begin{eqnarray}
\label{lamb0}
\|\widetilde{u}\|_{W^{s,p}((0,T)\times {\bf R}^{3})}&=&(\int_{\bf R}\int_{ {\bf R}^{3}}\int_{\bf R}
\int_{{\bf R}^{3}}\frac{|\widetilde{u(t_{1},{\bf x_{1}})}-\widetilde{u(t_{2},{\bf x_{2}})}|^{2}}
{[(t_{1}-t_{2})^{2}+({\bf x_{1}}-{\bf x_{2}})^{2}]^{\frac{4+sp}{2}}}dt_{1}d{\bf x_{1}}dt_{2}d{\bf x_{2}})^{1/p} \nonumber\\
&\leq& C_6 \| h \|_{L^{p}}\nonumber\\
&\leq & C_6( \| u \|_{L^{p}}+\| f \|_{L^{p}}),
\end{eqnarray}
where $C_6=\max(\frac{4}{3}\pi {R}^{3},T,C_4,C_5)$. 
Let $u_{\lambda}( t, {\bf x}, {\bf p})=u( \lambda t,\lambda {\bf x}, {\bf p})$ with $\lambda >0$.  
Then we have the following equalities:
\begin{eqnarray*}
\label{lamb1}
\|\widetilde{u_{\lambda}}\|_{W^{s,p}({\bf R}\times {\bf R}^{3})}&=&(\int_{\bf R}\int_{ {\bf R}^{3}}\int_{\bf R}\int_{ {\bf R}^{3}}
\frac{|\widetilde{u_{\lambda}(t_{1},{\bf x}_{1})}-\widetilde{u_{\lambda}(t_{2},{\bf x}_{2})}|^{2}}
{[(t_{1}-t_{2})^{2}+({\bf x}_{1}-{\bf x}_{2})^{2}]^{\frac{4+sp}{2}}}d{\bf x}_{1}dt_{1}d{\bf x}_{2}dt_{2})^{1/p}\\
&=&(\int_{\bf R}\int_{ {\bf R}^{3}}\int_{\bf R}\int_{ {\bf R}^{3}}
\frac{|\widetilde{u({\lambda}t_{1},{\lambda}{\bf x}_{1})}-\widetilde{u({\lambda}t_{2},{\lambda}{\bf x}_{2})}|^{2}}
{[(t_{1}-t_{2})^{2}+({\bf x}_{1}-{\bf x}_{2})^{2}]^{\frac{4+sp}{2}}}d{\bf x}_{1}dt_{1}d{\bf x}_{2}dt_{2})^{1/p}\\
&=&(\int_{\bf R}\int_{ {\bf R}^{3}}\int_{\bf R}\int_{ {\bf R}^{3}}
\frac{|\widetilde{u(t_{1},{\bf x}_{1})}-\widetilde{u(t_{2},{\bf x}_{2})}|^{2}}
{[(\frac{t_{1}-t_{2}}{\lambda})^{2}+(\frac{{\bf x}_{1}-{\bf x}_{2}}{\lambda})^{2}]^{\frac{4+sp}{2}}}
d(\frac{{\bf x}_{1}}{\lambda})d(\frac{t_{1}}{\lambda})d(\frac{{\bf x}_{2}}{\lambda})d(\frac{t_{2}}{\lambda}))^{1/p} 
\end{eqnarray*}
\begin{eqnarray}
&=&{\lambda}^{s-\frac{4}{p}}(\int_{\bf R}\int_{ {\bf R}^{3}}\int_{\bf R}\int_{ {\bf R}^{3}}
\frac{|\widetilde{u(t_{1},{\bf x}_{1})}-\widetilde{u(t_{2},{\bf x}_{2})}|^{2}}{[(t_{1}-t_{2})^{2}+({\bf x}_{1}-{\bf x}_{2})^{2}]^{\frac{4+sp}{2}}}
d{\bf x}_{1}dt_{1}d{\bf x}_{2}dt_{2})^{1/p},
\end{eqnarray}
\begin{eqnarray}
\label{lamb2}
\|u_{\lambda}\|_{L^{p}({\bf R}\times {\bf R}^{3}\times {\bf R}^{3})}
&=& (\int_{\bf R}\int_{ {\bf R}^{3}\times {\bf R}^{3}} |u( \lambda t,\lambda {\bf x}, {\bf p})|^{p}d{\bf p}d{\bf x}dt)^\frac{1}{p}\nonumber \\
&\leq & {\lambda}^{-\frac{4}{p}}\|u\|_{L^{p}({\bf R}\times {\bf R}^{3}\times {\bf R}^{3})},
\end{eqnarray}

\begin{eqnarray}
\label{lamb3}
\|\frac {\partial u_{\lambda}}{\partial t}+\frac {{\bf p}}{p_0}\frac{\partial u_{\lambda}}{\partial{\bf x}}\|_{L^{p}({\bf R}\times {\bf R}^{3}\times {\bf R}^{3})}
&=& \int_{\bf R}\int_{ {\bf R}^{3}\times {\bf R}^{3}} |\frac {\partial u_{\lambda}}{\partial t}+\frac {{\bf p}}{p_0}
\frac{\partial u_{\lambda}}{\partial{\bf x}}|^{p}d{\bf p}d{\bf x}dt \nonumber \\
&=& {\lambda}^{1-\frac{4}{p}}\|\frac {\partial u}{\partial t}+\frac {{\bf p}}{p_0}\frac{\partial u}{\partial{\bf x}}\|_{L^{p}({\bf R}\times {\bf R}^{3}\times {\bf R}^{3})}.
\end{eqnarray}
Replacing $u$ with $u_{\lambda}$ in (\ref{lamb0}) and using (\ref{lamb1}), (\ref{lamb2}) and (\ref{lamb3}),  
we can obtain 
\begin{eqnarray*}
\label{lamb4}
&&{\lambda}^{s-\frac{4}{p}}(\int_{\bf R}\int_{ {\bf R}^{3}}\int_{\bf R}\int_{ {\bf R}^{3}}
\frac{|\widetilde{u(t_{1},{\bf x}_{1})}-\widetilde{u(t_{2},{\bf x}_{2})}|^{2}}{[(t_{1}-t_{2})^{2}+({\bf x}_{1}-{\bf x}_{2})^{2}]^{\frac{4+sp}{2}}}
d{\bf x}_{1}dt_{1}d{\bf x}_{2}dt_{2})^{1/p}\\
&\leq&C_6({\lambda}^{-\frac{4}{p}} \| u \|_{L^{p}}+ {\lambda}^{1-\frac{4}{p}}\|
\frac {\partial u}{\partial t}+\frac {{\bf p}}{p_0}\frac{\partial u}{\partial{\bf x}}\|_{L^{p}})
\end{eqnarray*}
\begin{eqnarray}
&=&C_6({\lambda}^{-\frac{4}{p}} \| u \|_{L^{p}}+ {\lambda}^{1-\frac{4}{p}}\| f \|_{L^{p}}).
\end{eqnarray}
Setting $\lambda= \| u \|_{L^{p}}/\| f \|_{L^{p}}$, we have
\begin{eqnarray}
\|\widetilde{u}\|_{W^{s,p}((0,T)\times {\bf R}^{3})}\leq C_6 \| u \|_{L^{p}}^{1-s} \| f \|_{L^{p}}^{s}.
\end{eqnarray}
\noindent This completes the proof of Theorem \ref{th1}.

\section*{Acknowledgement} 
This work was supported by NSFC 11171356. 
The authors would like to thank the referees of this paper for their helpful suggestions on this work.



\begin{appendix}
\section{The proof of Lemma  \ref{le2}}
\label{app}
In order to prove Lemma \ref{le2}, we need two estimates related to a subset in the ball $B_{R}$. First, we have the following lemma:
\begin{lemma}[\cite{j97}]
\label{le3}
For any given $\varepsilon>0 $, if $e^{\prime}\in {\bf R}$, ${\bf e}\in {\bf R}^{3}$, and 
$|{\bf e}|^{2}+e^{\prime 2}=1$, then for all $R>0$,  there exists a constant $C_{R}$ which only depends on $R$, such that
\begin{equation}
\label{le31}
{\sf mes}(E_{R})\leq C_{R}{\varepsilon},
\end{equation}
where $E_{R}=\{{\bf p}\in B_{R}: -\varepsilon-e^{\prime}\leq\frac{{\bf p}\cdot{\bf e}}{\sqrt{|{\bf p}|^{2}+1}}\leq\varepsilon-e^{\prime}\}$.
\end{lemma}

\begin{proof}
Since we know that
\begin{equation}
\label{ep1}
{\sf mes}(E_{R})\leq {\sf mes}(B_{R})=\frac{4}{3}{\pi}{R}^{3}\frac{\varepsilon}{\varepsilon}\leq \frac{8}{3}{\pi}{R}^{3}{\varepsilon}
\end{equation}
for ${\varepsilon}\geq 1/2$, in order to show (\ref{le31}), 
it suffices to consider the case when $0<{\varepsilon}<1/2$.

For any ${\varepsilon}\in (0,1/2)$,  since ${\bf p}\in B_{R}$,  we can show 
that $|\frac{{\bf p}\cdot{\bf e}}{\sqrt{|{\bf p}|^{2}+1}}+e^{\prime} |> \varepsilon$ if 
we have the following inequality:
\begin{equation}
\label{ep2}
|e^{\prime} |-R|{\bf e}|> \varepsilon.
\end{equation}
Because $|{\bf e}|^{2}+e^{\prime 2}=1$, (\ref{ep2}) is equal to $\sqrt{1-|{\bf e}|^{2}}-R|{\bf e}|> \varepsilon$ which implies that

\begin{equation}
\label{ep3}
|{\bf e}|< \frac{-R{\varepsilon}+\sqrt{R^{2}+1-{\varepsilon}}}{R^{2}+1}.
\end{equation}

\noindent Therefore we have 

\begin{equation}
\label{ep4}
|{\bf e}|\geq \frac{-R{\varepsilon}+\sqrt{R^{2}+1-{\varepsilon}}}{R^{2}+1}
\end{equation}

\noindent when $|\frac{{\bf p}\cdot{\bf e}}{\sqrt{|{\bf p}|^{2}+1}}+e^{\prime}|\leq \varepsilon$.
We introduce the following rotation transform:
\begin{equation}
\label{rotr}
\phi(e_{1}, e_{2}, e_{3})^{T}=(|{\bf e}|, 0, 0),
\end{equation}
thus getting
\begin{eqnarray*}
\label{mes}
{\sf mes}(E_{R})&=&{\sf mes}(\{{\bf p}\in B_{R}: -\varepsilon-e^{\prime}\leq\frac{{\bf p}\cdot{\bf e}}{\sqrt{|{\bf p}|^{2}+1}}\leq\varepsilon-e^{\prime}\})
\nonumber\\
&=&{\sf mes}(\{\phi{\bf p}\in B_{R}: -\varepsilon-e^{\prime}\leq\frac{(\phi{\bf p})\cdot(\phi{\bf e})}{\sqrt{|\phi{\bf p}|^{2}+1}}\leq\varepsilon-e^{\prime}\})
\nonumber\\
&=&{\sf mes}(\{\phi{\bf p}\in B_{R}: -\varepsilon-e^{\prime}\leq\frac{(\phi{\bf p})\cdot|{\bf e}|}{\sqrt{|\phi{\bf p}|^{2}+1}}\leq\varepsilon-e^{\prime}\})
\nonumber\\
&\leq& 4R^{2}{\sf mes}(\{(\phi{\bf p})_{1}: \phi{\bf p}\in B_{R} \hbox{ and } -\varepsilon-e^{\prime}\leq\frac{(\phi{\bf p})_{1}|{\bf e}|}
{\sqrt{|\phi{\bf p}|^{2}+1}}\leq\varepsilon-e^{\prime}\})
\end{eqnarray*}
\begin{eqnarray*}
\leq 4R^{2}{\sf mes}(\{(\phi{\bf p})_{1}: \phi{\bf p}\in B_{R} \hbox{ and } \frac{-\varepsilon-e^{\prime}}{|{\bf e}|}
\sqrt{1+R^{2}}\leq\frac{(\phi{\bf p})_{1}|{\bf e}|}{\sqrt{|\phi{\bf p}|^{2}+1}}\leq\frac{\varepsilon-e^{\prime}}{|{\bf e}|}\sqrt{1+R^{2}}\})
\end{eqnarray*}
\begin{eqnarray}
&\leq& 8R^{2}\sqrt{1+R^{2}}\frac{\varepsilon}{|{\bf e}|}\nonumber\\
&\leq& 8R^{2}\sqrt{1+R^{2}}{\varepsilon}{|\frac{R^{2}+1}{\sqrt{R^{2}+1-{\varepsilon}} - R{\varepsilon}}|}\nonumber\\
&\leq& 16R(1+R^{2})^{3/2}{\varepsilon}.
\end{eqnarray}
Put $C_{R}=\max(\frac{8}{3}{\pi}{R}^{3}, 16R(1+R^{2})^{3/2})$. Then, by (\ref{ep1}) and (\ref{mes}), (\ref{le31}) holds. 
This completes the proof of Lemma  \ref{le3}.
\end{proof}

Then we have the following estimate of an integral:
\begin{lemma}[\cite{j97}]
\label{le4}
For any given $\varepsilon>0 $, $e^{\prime}\in {\bf R}$, ${\bf e}\in {\bf R}^{3}$, and $|{\bf e}|^{2}+e^{\prime 2}=1$, then for all $R>0$, 
there exists a constant $C_{R}$ which only depends on $R$,  such that
\begin{equation}
\label{le4ineq}
\int_{B_{R}(|\frac{{\bf p}\cdot{\bf e}}{\sqrt{|{\bf p}|^{2}+1}}+e^{\prime}|>\varepsilon)}\frac{1}
{|\frac{{\bf p}\cdot{\bf e}}{\sqrt{|{\bf p}|^{2}+1}}+e^{\prime}|^{2}}d{\bf{p}}\leq\frac{2C_{R}}{\varepsilon},
\end{equation}
where $C_{R}$ is equal to that given in Lemma \ref{le3}.
\end{lemma}

\begin{proof}
Notice that the integral on the right of (\ref{le4ineq}) can be rewritten as 
\begin{eqnarray}
\label{mes1}
&& \int_{B_{R}(|\frac{{\bf p}\cdot{\bf e}}{\sqrt{|{\bf p}|^{2}+1}}+e^{\prime}|>\varepsilon)}\frac{1}
{|\frac{{\bf p}\cdot{\bf e}}{\sqrt{|{\bf p}|^{2}+1}}+e^{\prime}|^{2}}d{\bf{p}}\nonumber\\
&=&  2 \int_{B_{R}(|\frac{{\bf p}\cdot{\bf e}}{\sqrt{|{\bf p}|^{2}+1}}+e^{\prime}|>\varepsilon)}
\int_{ |\frac{{\bf p}\cdot{\bf e}}{\sqrt{|{\bf p}|^{2}+1}}+e^{\prime} | }^{+\infty}\frac{1}{t^{3}}d{t}d{\bf p}\nonumber\\
&=&  2  \int_{B_{R}(|\frac{{\bf p}\cdot{\bf e}}{\sqrt{|{\bf p}|^{2}+1}}+e^{\prime}|>\varepsilon)}
(\int_{\varepsilon}^{+\infty}t^{-3}\chi_{\{|\frac{{\bf p}\cdot{\bf e}}{\sqrt{|{\bf p}|^{2}+1}}+e^{\prime}|, +\infty\}}(t)d{t})d{\bf p}\nonumber\\
&=&  2 \int_{\varepsilon}^{+\infty} t^{-3} ( \int_{B_{R}
(|\frac{{\bf p}\cdot{\bf e}}{\sqrt{|{\bf p}|^{2}+1}}+e^{\prime}|>\varepsilon)}
\chi_{\{|\frac{{\bf p}\cdot{\bf e}}{\sqrt{|{\bf p}|^{2}+1}}+e^{\prime}|\leq t\}}({\bf p})d{\bf p})d{t},
\end{eqnarray}
where $\chi_{\{|\frac{{\bf p}\cdot{\bf e}}{\sqrt{|{\bf p}|^{2}+1}}+e^{\prime}|, +\infty\}}(t)$ 
and $\chi_{\{|\frac{{\bf p}\cdot{\bf e}}{\sqrt{|{\bf p}|^{2}+1}}+e^{\prime}|\leq t\}}({\bf p})$ 
are characteristic functions.
It follows from Lemma \ref{le3} that
\begin{eqnarray}
\label{mes2}
 \int_{B_{R}(|\frac{{\bf p}\cdot{\bf e}}{\sqrt{|{\bf p}|^{2}+1}}+e^{\prime}|>\varepsilon)}
\chi_{\{|\frac{{\bf p}\cdot{\bf e}}{\sqrt{|{\bf p}|^{2}+1}}+e^{\prime}|\leq t\}}({\bf p})d{\bf p}\nonumber\\
\leq {\sf mes} ({\bf p}\in B_{R}:|\frac{{\bf p}\cdot{\bf e}}{\sqrt{|{\bf p}|^{2}+1}}+e^{\prime}|\leq t ) \leq C_{R}t.
\end{eqnarray}
Putting (\ref{mes2}) into (\ref{mes1}), we can get 
\begin{eqnarray}
\label{mes3}
&& \int_{B_{R}(|\frac{{\bf p}\cdot{\bf e}}{\sqrt{|{\bf p}|^{2}+1}}+e^{\prime}|>\varepsilon)}
\frac{1}{|\frac{{\bf p}\cdot{\bf e}}{\sqrt{|{\bf p}|^{2}+1}}+e^{\prime}|^{2}}d{\bf{p}}\nonumber\\
&\leq & 2C_{R}\int_{\varepsilon}^{+\infty} t^{-2}d{t} \leq 2C_{R}/{\varepsilon}.
\end{eqnarray}
\end{proof}

\begin{proof}[Proof of Lemma \ref{le2}]
We first consider $\|\widetilde{u}\|_{H^{\frac{1}{2}}((0,T)\times {\bf R}^{3})}$. 
Since ${\sf supp} u\subset [\varepsilon_{0},T-\varepsilon_{0}]\times {\bf R}^{3}\times B_{R}$,
we extend the domain of $u=u(t, {\bf x}, {\bf p})$ from $(0,T)\times {\bf R}^{3}\times {\bf R}^{3}$
into ${\bf R}\times {\bf R}^{3}\times {\bf R}^{3}$ such that $u=0$ outside $(0,T)\times {\bf R}^{3}\times {\bf R}^{3}$
and then define $\widehat{u}(\tau, {\bf z}, {\bf p})$ as the Fourier transformation of $u(t, \bf x, \bf p)$ with respect to $t$ and ${\bf x}$, 
that is, 
\begin{equation}
\label{fou}
\widehat{u}(\tau, {\bf z}, {\bf p})=\left(\frac{1}{\sqrt{2\pi}}\right)^4\int\int_{{\bf R}\times {\bf R}^{3}}u(t, {\bf x},{ \bf p})e^{-i\tau t-i{\bf x}\cdot {\bf z}}d{\bf x}d{\bf t}.
\end{equation}
Because $u(t,\bf x,\bf p)$ and $\frac {\partial u}{\partial t}+\frac {{\bf p}}{p_0}
\frac{\partial u}{\partial{\bf x}}(t,{\bf x},{\bf p} )$ both belong to $ L^{2}({\bf R}\times {\bf R}^{3}\times {\bf R}^{3})$,  using Plancherel's identity, we know that 
$\widehat{u}$ and $(\tau+\frac{{\bf p}\cdot \bf {z}}{\sqrt{1+{\bf p}^{2}}})\widehat{u}$ both belong to 
$ L^{2}({\bf R}\times {\bf R}^{3}\times {\bf R}^{3})$ too.
By the definition of the norm of $\widetilde{u}$ in $H^{\frac{1}{2}}({\bf R}\times {\bf R}^{3})$, we know that 
\begin{eqnarray}
\label{normh2}
\|\widetilde{u}\|_{H^{\frac{1}{2}}(R\times {\bf R}^{3})}=(\int_{{\bf R}\times {\bf R}^{3}}|{\bf z}^{2}+\tau^{2}|^{1/2} |
\int_{{\bf R}^{3}}\widehat{u}(\tau, {\bf z},{\bf p})d{\bf p}|^{2}d{\bf z}d{\tau})^{1/2}.
\end{eqnarray} 
It is worth mentioning that (\ref{normh1}) and (\ref{normh2}) are equivalent norms (see proposition 1.37 in \cite{bcd}).
Then we estemate the integral in (\ref{normh2}). For every $\alpha>0$, we have 
\begin{eqnarray}
\label{1}
\|\widetilde{u}\|_{H^{\frac{1}{2}}(R\times {\bf R}^{3})}^{2}&=&\int_{{\bf R}\times {\bf R}^{3}}| {\bf z}^{2}+\tau^{2}|^{1/2}
 |\int_{{\bf R}^{3}}\widehat{u}(\tau, {\bf z}, {\bf p})d{\bf p}|^{2}d{\bf z}d{\tau}\nonumber\\
&=&\int_{{\bf R}\times {\bf R}^{3}}| {\bf z}^{2}+\tau^{2}|^{1/2} |\int_{|\frac{{\bf p}\cdot {\bf z}}{\sqrt{1+{\bf p}^{2}}}+\tau|\leq\alpha}
\widehat{u}(\tau, {\bf z},{\bf p})d{\bf p}|^{2}d{\bf z}d{\tau}\nonumber\\
&+&\int_{{\bf R}\times {\bf R}^{3}}| {\bf z}^{2}+\tau^{2}|^{1/2} |\int_{|\frac{{\bf p}\cdot {\bf z}}{\sqrt{1+{\bf p}^{2}}}+\tau|>\alpha}
\widehat{u}(\tau, {\bf z},{\bf p})d{\bf p}|^{2}d{\bf z}d{\tau} \nonumber\\
&\triangleq & I_{1}+ I_{2},
\end{eqnarray}
where, 
\begin{eqnarray}
\label{i1}
I_{1}&=&\int_{{\bf R}\times {\bf R}^{3}}| {\bf z}^{2}+\tau^{2}|^{1/2} |\int_{|\frac{{\bf p}\cdot {\bf z}}{\sqrt{1+{\bf p}^{2}}}+\tau|\leq\alpha}
\widehat{u}(\tau, {\bf z},{\bf p})d{\bf p}|^{2}d{\bf z}d{\tau},
\end{eqnarray}
\begin{eqnarray}
\label{i2}
I_{2}&=&\int_{{\bf R}\times {\bf R}^{3}}| {\bf z}^{2}+\tau^{2}|^{1/2} |\int_{|\frac{{\bf p}\cdot {\bf z}}{\sqrt{1+{\bf p}^{2}}}+\tau|>\alpha}
\widehat{u}(\tau, {\bf z},{\bf p})d{\bf p}|^{2}d{\bf z}d{\tau}.
\end{eqnarray}
We below estimate $I_{1}$ and $I_{2}$ respectively. 
For $I_{1}$, using the Cauchy-Schwarz inequality, we have
\begin{eqnarray}
\label{i1_1}
I_{1}&=&\int_{{\bf R}\times {\bf R}^{3}}|{\bf z}^{2}+\tau^{2}|^{1/2} |\int_{|\frac{{\bf p}\cdot{\bf z}}{\sqrt{1+{\bf p}^{2}}}+\tau|\leq\alpha}
\widehat{u}(\tau, {\bf z},{\bf p})d{\bf p}|^{2}d{\bf z}d{\tau}\nonumber\\
&\leq& \int_{{\bf R}\times {\bf R}^{3}}|{\bf z}^{2}+{\tau}^{2}|^{1/2}(\int_{|\frac{{\bf p}\cdot{\bf z}}
{\sqrt{1+{\bf p}^{2}}}+\tau|\leq\alpha}d{\bf p})(\int_{|\frac{{\bf p}\cdot{\bf z}}{\sqrt{1+{\bf p}^{2}}}+\tau|\leq\alpha}|\widehat{u}|^{2}d{\bf p})d{\bf z}d{\tau}.
\end{eqnarray}
From Lemma \ref{le3}, we obtain
\begin{eqnarray}
\label{i1_2}
|{\bf z}^{2}+{\tau}^{2}|^{1/2}\int_{|\frac{{\bf p}\cdot{\bf z}}{\sqrt{1+{\bf p}^{2}}}+\tau|\leq\alpha}d{\bf p}\leq C_{R}{\alpha}.
\end{eqnarray}
It follows from (\ref{i1_1}) and (\ref{i1_2}) that
\begin{equation}
\label{i1_3}
I_{1}\leq C_{R}{\alpha} \int_{{\bf R}\times {\bf R}^{3}\times {\bf R}^{3}}|\widehat{u}|^{2}d{\bf p}d{\bf z}d{\tau}.
\end{equation}

Similarly, applying the Cauchy-Schwarz inequality again, we have
\begin{eqnarray*}
\label{i2_1}
I_{2}&=&\int_{{\bf R}\times {\bf R}^{3}}|{\bf z}^{2}+\tau^{2}|^{1/2} |\int_{|\frac{{\bf p}\cdot{\bf z}}{\sqrt{1+{\bf p}^{2}}}+\tau|>\alpha}
\widehat{u}(\tau, {\bf z},{\bf p})d{\bf p}|^{2}d{\bf z}d{\tau}
\end{eqnarray*}
\begin{eqnarray}
\leq\int_{{\bf R}\times {\bf R}^{3}}(\int_{|\frac{{\bf p}\cdot{\bf z}}{\sqrt{1+{\bf p}^{2}}}+\tau|>\alpha}
\frac{|{\bf z}^{2}+{\tau}^{2}|^{1/2}}{|\frac{{\bf p}\cdot{\bf z}}{\sqrt{1+{\bf p}^{2}}}+\tau|^{2}}d{\bf p})
(\int_{|\frac{{\bf p}\cdot{\bf z}}{\sqrt{1+{\bf p}^{2}}}+\tau|>\alpha}|\frac{{\bf p}\cdot{\bf z}}
{\sqrt{1+{\bf p}^{2}}}+\tau|^{2}|\widehat{u}|^{2}d{\bf p})d{\bf z}d{\tau}.
\end{eqnarray}
From Lemma \ref{le4}, we obtain
\begin{eqnarray}
\label{i2_2}
 &&\int_{{\bf R}\times {\bf R}^{3}}\int_{|\frac{{\bf p}\cdot{\bf z}}{\sqrt{1+{\bf p}^{2}}}+\tau|>\alpha}\frac{|{\bf z}^{2}+{\tau}^{2}|^{1/2}}{|\frac{{\bf p}\cdot{\bf z}}
{\sqrt{1+{\bf p}^{2}}}+\tau|^{2}}d{\bf p}d{\bf z}d{\tau}\nonumber\\
 &=&\int_{{\bf R}\times {\bf R}^{3}}\frac{1}{\sqrt{{\bf z}^{2}+{\tau}^{2}}}(\int_{|(\frac{{\bf p}\cdot{\bf z}}
{\sqrt{1+{\bf p}^{2}}}+\tau)/{\sqrt{{\bf z}^{2}+{\tau}^{2}}}|>\alpha/{\sqrt{{\bf z}^{2}+{\tau}^{2}}}}
\frac{1}{|(\frac{{\bf p}\cdot{\bf z}}{\sqrt{1+{\bf p}^{2}}}+\tau)/{\sqrt{{\bf z}^{2}+{\tau}^{2}}}|^{2}}d{\bf p})d{\bf z}d{\tau}\nonumber\\
& \leq&\frac{1}{\sqrt{{\bf z}^{2}+{\tau}^{2}}}\cdot2C_{R}\cdot\frac{\sqrt{{\bf z}^{2}+{\tau}^{2}}}{\alpha}=\frac{2C_{R}}{\alpha}.
\end{eqnarray}
By (\ref{i2_1}) and (\ref{i2_2}), it follows that
\begin{equation}
\label{i2_3}
I_{2}\leq\frac{2C_{R}}{\alpha}\int_{{\bf R}\times {\bf R}^{3}\times {\bf R}^{3}}|
\frac{{\bf p}\cdot{\bf z}}{\sqrt{1+{\bf p}^{2}}}+\tau|^{2}|\widehat{u}|^{2}d{\bf p}d{\bf z}d{\tau}.
\end{equation}

Hence, by (\ref{1}), (\ref{i1_3}) and (\ref{i2_3}),  we know that
\begin{eqnarray*}
\label{2}
\|\widetilde{u}\|_{H^{\frac{1}{2}}({\bf R}\times {\bf R}^{3})}^{2}&=&\int_{{\bf R}\times {\bf R}^{3}}|{\bf z}^{2}+\tau^{2}|^{1/2} |
\int_{{\bf R}^{3}}\widehat{u}(\tau, {\bf z},{\bf p})d{\bf p}|^{2}d{\bf z}d{\tau}\\
&\leq& C_{R}{\alpha} \int_{{\bf R}\times {\bf R}^{3}\times {\bf R}^{3}}|\widehat{u}|^{2}d{\bf p}d{\bf z}d{\tau}
+\frac{2C_{R}}{\alpha}\int_{{\bf R}\times {\bf R}^{3}\times {\bf R}^{3}}|\frac{{\bf p}\cdot{\bf z}}
{\sqrt{1+{\bf p}^{2}}}+\tau|^{2}|\widehat{u}|^{2}d{\bf p}d{\bf z}d{\tau}
\end{eqnarray*}
\begin{eqnarray}
\leq 2C_{R}({\alpha} \int_{{\bf R}\times {\bf R}^{3}\times {\bf R}^{3}}|\widehat{u}|^{2}d{\bf p}d{\bf z}d{\tau}+\frac{1}{\alpha}
\int_{{\bf R}\times {\bf R}^{3}\times {\bf R}^{3}}|\frac{{\bf p}\cdot\bf{z}}{\sqrt{1+{\bf p}^{2}}}+\tau|^{2}|\widehat{u}|^{2}d{\bf p}d{\bf z}d{\tau}).
\end{eqnarray}
In (\ref{2}), we set  
\begin{eqnarray}
\label{alpha}
\alpha=( \int_{{\bf R}\times {\bf R}^{3}\times {\bf R}^{3}}|\widehat{u}|^{2}d{\bf p}d{\bf z}d{\tau})^{1/2}
(\int_{{\bf R}\times {\bf R}^{3}\times {\bf R}^{3}}|\frac{{\bf p}\cdot{\bf z}}{\sqrt{1+{\bf p}^{2}}}+\tau|^{2}|\widehat{u}|^{2}d{\bf p}d{\bf z}d{\tau})^{1/2},
\end{eqnarray}
thus getting 
\begin{eqnarray}
\label{3}
\|\widetilde{u}\|_{H^{\frac{1}{2}}({\bf R}\times {\bf R}^{3})}^{2}&=&\int_{{\bf R}\times {\bf R}^{3}}|{\bf z}^{2}+\tau^{2}|^{1/2} |
\int_{{\bf R}^{3}}\widehat{u}(\tau, {\bf z},{\bf p})d{\bf p}|^{2}d{\bf z}d{\tau}\nonumber\\
&\leq& 2C_R( \int_{{\bf R}\times {\bf R}^{3}\times {\bf R}^{3}}|\widehat{u}|^{2}d{\bf p}d{\bf z}d{\tau})^{1/2}(\int_{{\bf R}\times {\bf R}^{3}\times {\bf R}^{3}}|
\frac{{\bf p}\cdot{\bf z}}{\sqrt{1+\bf{p}^{2}}}+\tau|^{2}|\widehat{u}|^{2}d{\bf p}d{\bf z}d{\tau})^{1/2}\nonumber\\
&=& 2C_R \|u\|_{L^{2}}\| \frac {\partial u}{\partial t}+\frac {{\bf p}}{p_0}\frac{\partial u}{\partial{\bf x}}(t,{\bf x},{\bf p})\|_{L^{2}}.
\end{eqnarray}
Therefore $\widetilde{u}=\int_{{\bf R}^{3}}u(t, {\bf x}, {\bf p})d{\bf p}$ belongs to $H^{\frac{1}{2}}(R\times {\bf R}^{3})$.
Since (\ref{normh1}) and (\ref{normh2})  are equivalent, we know from (\ref{3}) that
\begin{eqnarray}
\|\widetilde{u}\|_{H^{\frac{1}{2}}((0,T)\times {\bf R}^{3})}&=&\int_{{\bf R}\times {\bf R}^{3}}\int_{{\bf R}\times {\bf R}^{3}}
\frac{|\widetilde{u(t_{1},{\bf x}_{1})}-\widetilde{u(t_{2},{\bf x}_{2})}|^{2}}
{[(t_{1}-t_{2})^{2}+({\bf x}_{1}-{\bf x}_{2})^{2}]^{\frac{5}{2}}}dt_{1}d{\bf x}_{1}dt_{2}d{\bf x}_{2}\nonumber\\
&\leq & \sqrt{2}C_R^{1/2}\| u \|_{L^{2}}^{1/2} \|f\|_{L^{2}}^{1/2}.
\end{eqnarray}
This completes our proof of Lemma \ref{le2}.
\end{proof}

\end{appendix}


\begin{thebibliography}{99}
\bibitem[1]{a96}Andreasson H., Regularity of the gain term and strong L1 convergence to the equilibrium 
for the relativistic Boltzmann equation, {Anal.} {Math.}, 1996, {\bf 27}: 1386-1405.
\bibitem[2]{aci04}Andreasson H., Calogero S., Illner R., On blowup  for gain-term-only 
classical and relativistic Boltzmann equations, {App.} {Math.}, 2004, {\bf 27}: 2231-2240.
\bibitem[3]{bcd}Bachouri H., Chemin J., Danchin R., Fourier analysis and 
nonlinear partial differential equations, Springer-Verlag, Berlin Heidelberg, 2011.
\bibitem[4]{b}Bichteler K., On the Cauchy problem of the relativistic  
Boltzmann equations, {Commun.} {Math.} {Phys.}, 1967, {\bf 4}: 352-364.
\bibitem[5]{bl}Bergh J.,  Lofstrom J., Interpolation Spaces, Springer-Verlag, New York, Berlin, 1976.
\bibitem[6]{ck}Cercignani C., Kermer G., The relativistic Boltzmann equation 
theory and application, Birkhaeuser, Boston£¬2002.
\bibitem[7]{fl1}Debbasch F., Leeuwen {W.} {A.} {V.}, General relativistic 
Boltzmann equation, \uppercase\expandafter{\romannumeral1},  
Physica A: Statistical Mechanics \& Its Applications, 2009, {\bf 388}: 1079-1104.
\bibitem[8]{fl2}Debbasch F., Leeuwen {W.} {A.} {V.}, General relativistic 
Boltzmann equation, \uppercase\expandafter{\romannumeral2},  
Physica A: Statistical Mechanics \& Its Applications, 2009, {\bf 388}: 1818-1834.
\bibitem[9]{dhm}Denicol {G.} {S.}, Ulrich {H.}, Mauricio {M.}, Jorge {N.}, Michael {S.}, 
A new exact solution of the relativistic Boltzmann 
equation and its hydrodynamic limit, Physical Review Letters, 2014, {\bf 113}(20): 202301-202301. 
\bibitem[10]{dl}DiPerna {R.} {J}., Lions {P.} {L.}, On the Cauchy problem for  Boltzmann
equations: Global existence and weak stability, {Ann.} {Math.}, 1989, {\bf 130}: 321-366.
\bibitem[11]{pz}Dolan P., A construction of the general relativistic Boltzmann equation,  {Math.} , 1984, {\bf 7}: 591-597.
\bibitem[12]{de85}Dudy\'{n}ski M., Ekiel-Je\.{z}ewska {M.} {L.},  Errata: Causality of the
Linearized Relativistic Boltzmann Equation, {Phys.} {Rev.} {Lett.}, 1985, {\bf 56}: 2228-2228.
\bibitem[13]{de88}Dudy\'{n}ski M., Ekiel-Je\.{z}ewska {M.} {L.}, On  the  Linearized Relativistic Boltzmann Equation,
{Commun.} {Math.} {Phys.}, 1988, {\bf 115}: 607-629.
\bibitem[14]{de92}Dudynski M., Ekiel-Jezewska {M.} {L.}, Global existence proof 
for the relativistic Boltzmann equation, Journal of Statistical Physics, 1992, {\bf 66}: 991-1001.
\bibitem[15]{emv}Esobedo M., Michler S., Valle {M.} {A.}, Homogeneous 
Boltzmann equation in quantum relativistic kinetic theory, Electronic Journal of differential equations, 2003, {\bf 4}: 1-85.
\bibitem[16]{gs93}Glassey R., Strauss {W.} {A.}, Asymptotic stability of the relativistic 
Maxwellian, {Publ.} {Res.} {Inst.} {Math.} {Sci.}, 1993, {\bf 29}: 301-347.
\bibitem[17]{gs95}Glassey R., Strauss {W.} {A.}, Asymptotic stability of the relativistic 
Maxwellian via fourteen moments, {Stat.} {Phys.} , 1995, {\bf 24}: 657-678.
\bibitem[18]{g06}Glassey R., Global solutions to the Cauchy problem for the relativistic Boltzmann equation with near-vacuum
data, {Commun.} {Math.} {phys.}, 2006, {\bf 264}: 705-724.
\bibitem[19]{gl} Golse F., Lions {P.} {L.}, Perthame B., Sentis R., Regularity of
 the moments of the solution of a transport equation, Journal of
Functional Analysis, 1988, {\bf 76}: 110-125.
\bibitem[20]{gv}de Groot {S.} {R.},  Van Leeuwen {W.} {A.},  Van  Weert {Ch.} {G.}, Relativistic Kinetic
Theory, North-Holland, Amsterdam, 1980.
\bibitem[21]{hk}Ha S., Kim Y., Lee H., Noh S., Asympotic completeness 
for relativistic kinetic equations with short-range interaction forces, {App.} {Ann.}, 2007, {\bf 14}: 251-262.
\bibitem[22]{hy}Hsiao L., Yu H., Asympotic stability of relativistic Maxwellian, {Math.} {App.} {Sci.}, 2006, {\bf 29}: 1481-1499.
\bibitem[23]{i}Israel W., Relativistic kinetic theory of simple gas, {Commun.} {Math.} {phys.}, 1963, {\bf 4}: 1163-1181.
\bibitem[24]{lm}Lichnerowicz A., Marrot R., Propri\'{e}t\'{e}s statistiques des ensembles de particules en relativit\'{e} restreinte,  
{Compt.} {Rend.} {Acad.} {Sci.} (Paris), 1940, {\bf 210}: 759-761.
\bibitem[25]{j97}Jiang Z., Compactness Related to the Transport Operator 
of the Relaivistic Boltzmann Equation,  Acta Mathematica Scientia (in Chinese), 1997, {\bf 17}(3): 330-335.
\bibitem[26]{j}Jiang Z.,  On the relatvistic Boltzmann equation, Acta Mathematica Scientia, 1998, {\bf 18}: 348-360.
\bibitem[27]{j98}Jiang Z.,  Existence of global solution to the Cauchy problem 
for the relativistic Boltzmann equation in a periodic box, Acta Mathematica Scientia, 1998, {\bf 18}: 375-384.
\bibitem[28]{j99}Jiang Z.,  On Cauchy problem for the relatvistic Boltzmann 
equationin a periodic box: global existence, Transport theory and statistical physics, 1999, {\bf 28}: 617-628.
\bibitem[29]{j07}Jiang Z.,  Global existence proof for relativistic Boltzmann equation 
with hard interactions, {Stat.} {Phys}, 2008, {\bf 130}: 535-544.
\bibitem[30]{ju}J$\ddot{u}$ttner F., Das Maxwellsche Gesetz der Geschwindigkeitsverteilung 
in der Relativitheorie,  {Ann.} {Phys.}, 1911, {\bf 34}: 856-882.
\bibitem[31]{r2004}Rein G., Global weak solutions to the relativistic Vlasov-Maxwell system-Revisited,  {Comm.} {Math.} {Sci.}, 2004, {\bf 2}: 145-158.
\bibitem[32]{s2010a}Strain {R.} {M.}, Global Newtonian limit for the relativistic  Boltzmann equation near vacuum, 
SIAM Journal on Mathematical Analysis, 2010, {\bf 42}(4):1568-1601.
\bibitem[33]{s2010b}Strain {R.} {M.}, Asymptotic stability of the relativistic Boltzmann equantion 
for the soft potentials, {Commun.} {Math.} {Phys.}, 2010, {\bf 300}(2): 529-597.
\bibitem[34]{s2011}Strain {R.} {M.}, Coordinates in the relativistic Boltzmann theory, 
Kinetic and Related Models, 2011{\bf 4}(1):345-359.
\bibitem[35]{s}Stewart J., Non-equilibrium relativistic kinetic theory, volume 10 of Lectures notes in physics, Springer-Verlag, Berlin, 1971.
\bibitem[36]{t}Triebel H., Interpolation Theory, Function Spaces, Differential Operators, North-Holland, Amsterdam, 1978.
\bibitem[37]{tk}Tsumura K., Kunihiro T., Derivation of relativistic hydrodynamic equation consitient with relativistic Boltzmann 
equation by renormalizatio-group method,  {Phys.}, 2012, {\bf 48}: 162-173.
\end{thebibliography}
\end{document}